\documentclass[10pt,draft,twoside]{amsart}
\usepackage{amssymb}
\usepackage{latexsym}
\usepackage{amsfonts}
\usepackage{amsmath}
\usepackage{amsthm}

\theoremstyle{definition}
\newtheorem{ntex}{Definition}
\newtheorem{cql}[ntex]{Definition}

\theoremstyle{remark}
\newtheorem{multques}[ntex]{Question}
\newtheorem{orderques}[ntex]{Question}
\newtheorem{disjointques}[ntex]{Question}

\theoremstyle{plain}
\newtheorem{alternating}[ntex]{Theorem}

\newtheorem{neighbours}[ntex]{Lemma}
\newtheorem{symmetric}[ntex]{Lemma}
\newtheorem{notfixed}[ntex]{Lemma}
\newtheorem{elusive1}[ntex]{Corollary}
\newtheorem{ntproduct}[ntex]{Lemma}
\newtheorem{prodsq}[ntex]{Corollary}
\newtheorem{sqneighbours}[ntex]{Lemma}

\newtheorem{pair2}[ntex]{Lemma}
\newtheorem{altrep}[ntex]{Proposition}
\newtheorem{neigh}[ntex]{Lemma}
\newtheorem{partition}[ntex]{Lemma}
\newtheorem{q3m4}[ntex]{Lemma}

\numberwithin{equation}{section}
\numberwithin{ntex}{section}

\DeclareMathOperator{\Prod}{Prod}
\DeclareMathOperator{\Aut}{Aut}
\DeclareMathOperator{\Rep}{Rep}
\DeclareMathOperator{\Diag}{Diag}
\DeclareMathOperator{\Pre}{Pre}

\renewcommand{\Gamma}{\varGamma}

\begin{document}      

\title{Elusive Codes in Hamming Graphs
}
\author{DANIEL R. HAWTIN, NEIL I. GILLESPIE and CHERYL E. PRAEGER}
\address{[Hawtin, Gillespie and Praeger] Centre for the Mathematics of Symmetry and Computation\\
School of Mathematics and Statistics\\
The University of Western Australia\\
35 Stirling Highway, Crawley\\
Western Australia 6009\\
[Praeger] also affiliated with King Abdulaziz University, Jeddah, Saudi Arabia.}



\begin{abstract}
We consider a \emph{code} to be a subset of the vertex set of a \emph{Hamming
graph}. We examine \emph{elusive pairs}, code-group pairs where the code is not
determined by knowledge of its \emph{set of neighbours}. We construct a new
infinite family of elusive pairs, where the group in question acts transitively
on the set of neighbours of the code. In our examples, we find that the
\emph{alphabet size} always divides the \emph{length} of the code, and prove
that there is no elusive pair for the smallest set of parameters for which this
is not the case. We also pose several questions regarding elusive pairs.
\end{abstract}



\thanks{{\it Date:} draft typeset \today
 }

\maketitle            

\section{Introduction and Motivation}\label{intro}

Given a group that fixes setwise the set of neighbours of a certain code, the
question of whether the group necessarily fixes the code setwise was considered by
the second and third authors in \cite{ntrcodes}.  In particular, they considered codes in a
\textit{Hamming graph} $\Gamma=H(m,q)$, in which each vertex, and in particular
each codeword, is an $m$-tuple with entries from a set $Q$ of size $q$. In this
context, a codeword $\alpha$ with a single symbol changed, that is a single
error introduced, corresponds to a vertex $\nu$ in $\Gamma$ adjacent to
$\alpha$. We refer to $\nu$ as a \textit{neighbour} of $\alpha$, and for a code
$C$, the \emph{set of neighbours of $C$}, denoted by $\Gamma_1(C)$, consists of all
vertices of $\Gamma$ which are not in $C$, but are adjacent to at least one element
of $C$. 

The group fixing $\Gamma_1(C)$ setwise is a subgroup, $G$, of the automorphism
group, $\Aut(\Gamma)$, of $\Gamma$. Whether $G$ fixes $C$ setwise depends on
certain parameters of the code. One such parameter is the \textit{minimum
distance}, $\delta$, defined to be the smallest distance in $\Gamma$ between
distinct codewords in $C$. In particular, by \cite[Theorem 1]{ntrcodes}, if $C$
is a code in $H(m,q)$ with $\delta\geq 3$ such that $G$ does not fix $C$ setwise, then
one of the following holds:
\begin{itemize}
\item[(1)] $\delta=4$, $q=2$ and $m$ is even,
\item[(2)] $\delta=3$, and $m(q-1)$ is even.
\end{itemize}

The paper \cite{ntrcodes} exhibits an infinite family of codes and groups with
the parameters of (1), but no examples for case (2) are given. The aim of this
paper is to provide infinitely many examples for case (2). All of our examples
have $m$ a multiple of $q$ and we pose several questions about the parameters
and properties of such codes. We make the following definition.

\begin{ntex}
Let $C$ be a code in $\Gamma=H(m,q)$ with minimum distance $\delta$, and let
$X\leq \Aut(\Gamma)$ such that $X$ fixes $\Gamma_1(C)$ setwise, but does not fix
$C$ setwise. Then we call $(C,X)$ an \emph{elusive pair}, with parameters
$(m,q,\delta)$.
\end{ntex}

The paper \cite{ntrcodes} contains no comment on elusive pairs with the
parameters of (2). However the discussion following \cite[Problem 11.1]{neilphd}
asks if there exist elusive pairs with $\delta=3$ and $m(q-1)$ even. We prove
the following theorem.

\begin{alternating}\label{alternating}
Let $q\geq 3$ and $m$ be divisible by $q$. Then there exists a code $C$ with
minimum distance $\delta=3$, and a group $X$ such that $(C,X)$ is an elusive
pair with parameters $(m,q,3)$, and $X$ is transitive on $\Gamma_1(C)$.
\end{alternating}

We prove further in Section~\ref{case} that there are no elusive pairs with
parameters $(4,3,3)$. The following question, however, remains unanswered.

\begin{multques}\label{multques}
Do there exist elusive pairs with parameters $(m,q,3)$ such that $m$ is not a
multiple of $q$? More generally we ask for a determination of the possible
parameters of elusive pairs.
\end{multques}

\subsection{Commentary and Further Questions.}
An assumption frequently made in coding theory is that during transmission of an
encoded message, the probability of an error occurring is independent of the
symbol sent, and its position in the message. In \cite{fpas, ntrcodes}, the second and
third authors introduce \textit{neighbour transitivity} as a group theoretic
analogue of this assumption. A code $C$ is defined to be \textit{neighbour
transitive} if there exists an $X\leq \Aut(\Gamma)$ that fixes setwise and acts
transitively on both $C$ and $\Gamma_1(C)$.

For the infinite family of examples from \cite[Section 5]{ntrcodes}, it is
shown, in that paper, that there exists a group $X$ such that $X$ is transitive
on the set of neighbours of the code. Moreover, it is contained in a larger code
which shares the same neighbour set, and is, in fact, $X$-neighbour transitive.
The construction we give in Section~\ref{ex1} produces an $X$-neighbour
transitive code, $C'$, with minimum distance $\delta'=2$, such that $C'$
contains a code, $C$, with $\delta=3$ and $\Gamma_1(C)=\Gamma_1(C')$. It follows
that $X$ acts transitively on $\Gamma_1(C)$, but does not fix $C$ setwise, and
thus $(C,X)$ is an elusive pair.

For each elusive pair $(C,X)$ in Section \ref{ex1} and \ref{ex2}, and for those
constructed in \cite[Section 5]{ntrcodes}, we observe that if $x\in X$ does not
fix $C$ setwise, then there is only one possibility for the image, $C^x$. Thus,
under $X$, $C$ has two images, the original code $C$ and $C^x$. 

\begin{orderques}\label{orderques}
Does there exist an elusive pair $(C,X)$ such that $C$ has more than two images
under $X$? More generally, if $r=|\{C^x\mid x\in X\}|$, what values of $r$ are
possible?
\end{orderques}

We also note that, in all examples mentioned thus far, not only is $X$
transitive on $\Gamma_1(C)$, but $C$ is $X_C$-neighbour transitive, where $X_C$
is the subgroup of $X$ fixing $C$ setwise. However this is not true in general.
In Section~\ref{except} we construct a family of elusive pairs $(C,X)$, where
$X_C$ is not transitive on $C$ and $X$ is not transitive on $\Gamma_1(C)$. In
fact, for almost all of these examples $(C,X)$ there is no larger group $X'$
such that $(C,X')$ is an elusive pair and $X'$ is transitive on $\Gamma_1(C)$
(see Proposition~\ref{altrep}).

Another interesting feature of the elusive pairs $(C,X)$ in Section \ref{ex1}
and \ref{ex2}, and also those in \cite{ntrcodes}, is that if $x\in X$ does not
fix $C$ setwise, then $C^x$ and $C$ are disjoint. The family of examples in
Section~\ref{except} do not have this property. However they are not
$X_C$-neighbour transitive.

\begin{disjointques}\label{disjointques}
If $(C,X)$ is an elusive pair and $C$ is $X_C$-neighbour transitive, is it true
that, for each $x\in X$, either $C^x=C$ or $C$ and $ C^x$ are disjoint?
\end{disjointques}

\section{Notation}\label{not}

\subsection{Hamming Graphs.}

Let $C$ be a code of ordered $m$-tuples over an alphabet, $Q$, of size $q$. The
Hamming graph, $\Gamma=H(m,q)$, has vertex set consisting of all $m$-tuples with
entries from $Q$, with an edge existing between $m$-tuples which differ in
exactly one position. The \textit{Hamming distance}, $d(\alpha,\beta)$, between
two vertices, $\alpha,\beta\in\Gamma$, is defined as the number of entries in
which the two vertices differ. For a code $C$, the minimum distance, $\delta$,
is defined as $\delta=\min\{d(\alpha,\beta) \mid \alpha, \beta \in C,
\alpha\neq\beta\}$. For a vertex $\alpha\in\Gamma$, we denote the set of
vertices which are distance $r$ from $\alpha$ by
$\Gamma_r(\alpha)=\{\beta\in\Gamma \mid d(\alpha,\beta)=r\}$. We call
$\Gamma_1(\alpha)$ the \textit{set of neighbours of $\alpha$}.

For a vertex $\alpha\in\Gamma$, define $d(\alpha,C)=\min\{d(\alpha,\beta) \mid
\beta\in C\}$. This allows us to define the \textit{covering radius}, $\rho =
\max\{d(\alpha,C)\mid\alpha\in\Gamma\}$, and for any $r\leq \rho$ we define
$\Gamma_r(C)=\{\alpha\in\Gamma \mid d(\alpha,C)=r\}$. We refer to $\Gamma_1(C)$
as the \textit{set of neighbours of $C$}. Note that if $\delta\geq 2$,
$\Gamma_1(C)=\cup_{\alpha\in C}\Gamma_1(\alpha)$.

The automorphism group of the Hamming graph, $\Aut(\Gamma)$, is the semi-direct
product $N\rtimes L$, where $N\cong S_q^m$ and $L\cong S_m$, see \cite[Theorem
9.2.1]{brouwer}. Let $g=(g_1,\dots,g_m)\in N$, $\sigma\in L$ and
$\alpha=(\alpha_1,\ldots,\alpha_m)\in\Gamma$. Then $g$ and $\sigma$ act on
$\alpha$ as follows: 
\begin{equation*}
\alpha^g =(\alpha_1^{g_1},\ldots,\alpha_m^{g_m}),\quad\text{and}\quad
\alpha^\sigma=(\alpha_{1^{\sigma^{-1}}},\ldots,\alpha_{m^{\sigma^{-1}}}).
\end{equation*}

The automorphism group of a code $C\subseteq\Gamma$, is defined to be the
setwise stabiliser of $C$ in $\Aut(\Gamma)$, and denoted by $\Aut(C)$.

\subsection{Permutation Codes.}\label{permcodes}

Let $Q=\{1,\ldots,q\}$ and $S_q$ be the symmetric group of $Q$. For any
permutation $g\in S_q$, we associate with it the vertex
$\alpha(g)=(1^g,\ldots,q^g)$ in $H(q,q)$. Furthermore, for $T\subseteq S_q$, we
define the \textit{permutation code} $C(T)$ to be $C(T)=\{\alpha(g)\mid g\in
T\}$.

Permutation codes were first studied in the 1970's, in particular by Blake,
Cohen and Deza in \cite{permcodes}, but have recently gained attention due to a
potential application in powerline communication, where information is
transmitted as a string of frequencies through existing electrical
infrastructure. This approach presents extra problems for us to consider. We
need the power output to remain as constant as possible, as well as there being
extra noise considerations to take into account. Permutation codes have been
suggested as a solution to both of these problems, see \cite{powerlinecomm, codedmod}. 
For an overview of the subject see \cite{Huczynska}. Bailey
gives a decoding algorithm for permutation codes generated by groups in
\cite{bailey}.

In \cite{blake}, Blake shows how to find the minimum distance of any permutation
code constructed from a \textit{sharply $k$-transitive} group. A group $G$
acting on a set $\Omega$ is sharply $k$-transitive if for any two ordered
$k$-tuples of distinct points, there is a unique element of $G$ mapping the first to the second. So
the identity element is the unique element fixing $k$ points, and thus, for
$g_1, g_2\in G$ with $g_1\neq g_2$, we have $d(\alpha(g_1),\alpha(g_2))\geq
q-k+1$. So $\delta\geq q-k+1$. For example, $S_q$ is sharply $(q-1)$-transitive,
and if we let $g=(1 2)\in S_q$ then $d(\alpha(1),\alpha(g))=2$, thus $C(S_q)$
has minimum distance 2. Also $A_q$ is sharply $(q-2)$-transitive, so $C(A_q)$
has minimum distance $3$. In the same paper, Blake also briefly outlines a
decoding algorithm for $C(A_q)$.

Let $i,j\in Q$, $i\neq j$ and $g\in S_q$, and define $\nu(\alpha(g),i,j)$ to be
the vertex in $H(q,q)$ with $k$-th entry given by
\begin{equation*}
\nu(\alpha(g),i,j)|_k=\left\{
\begin{array}{ll}
k^g & \text{if } k\neq i\\
j^g & \text{if } k=i
\end{array}\right. .
\end{equation*}
If $i\neq j$, then $\nu(\alpha(g),i,j)$ differs from $\alpha(g)$ at the $i$-th
entry only, and thus $\nu(\alpha(g),i,j)\in \Gamma_1(\alpha(g))$. Each of the
$q(q-1)$ neighbours of $\alpha(g)$ is of this form; there are $q$ choices for
$i$ and, given $i$, there are $q-1$ choices for $j^g\neq i^g$, and hence of $j$.

For $y\in S_q$, we denote $x_y=(y,\ldots,y)\in N$, and define
$\Diag_q(S_q)=\{x_y \mid y\in S_q\}\leq N$. Also, for $z\in S_q$ let $\sigma(z)$
be the permutation in the top group $L$, induced by $z$.  Let $g,y,z\in
S_q$, $i\neq j$, $x_y= (y,\ldots,y)\in \Diag_q(S_q)$ and $\sigma(z)\in L$. Then, by \cite[Lemmas 5.1.1 and 5.1.1.3]{neilphd},  
\begin{equation}\label{act}
\alpha(g)^{x_y\sigma(z)}=\alpha(z^{-1}gy),\quad\text{and}
\quad\nu(\alpha(g),i,j)^{x_y\sigma(z)}=\nu(\alpha(z^{-1}gy),i^z,j^z).
\end{equation}

\section{Elusive Pairs}\label{example}

In this section we construct the examples that contribute to the proof of
Theorem~\ref{alternating}, as well as showing that there is no elusive pair for
the smallest set of parameters where $q$ does not divide $m$.

\subsection{Example 1.}\label{ex1}

We show that $(C(A_q),\Diag_q(S_q)\rtimes L)$ is an elusive pair, with
parameters $(q,q,3)$. We begin by showing that the larger code, $C(S_q)$, with
minimum distance two, has the same neighbour set as $C(A_q)$.

\begin{neighbours}\label{same}
For distinct $i,j\in \{1,\ldots,q\}$ and $g\in A_q$, let $g'=(ij)g$. Then $g'\in
S_q\setminus A_q$ and $\nu(\alpha(g),i,j)=\nu(\alpha(g'),j,i)$. Thus
$\Gamma_1(C(A_q))=\Gamma_1(C(S_q\setminus A_q))=\Gamma_1(C(S_q))$.
\end{neighbours}

\begin{proof}
Clearly $g'\in S_q\setminus A_q$. We have
\begin{align*}
\nu(\alpha(g),i,j)|_k&=\left\{
\begin{array}{ll}
k^{g} & \text{if } k\neq i\\
j^{g} & \text{if } k=i
\end{array}\right. \\
&=\left\{
\begin{array}{ll}
k^{g'} & \text{if } k\neq i,j\\
i^{g'} & \text{if } k=i\\
i^{g'} & \text{if } k=j
\end{array}\right. \\
&=\left\{
\begin{array}{ll}
k^{g'} & \text{if } k\neq j\\
i^{g'} & \text{if } k=j
\end{array}\right.\\
&=\nu(\alpha(g'),j,i)|_k \text{ } .
\end{align*}
Thus $\nu(\alpha(g),i,j)=\nu(\alpha(g'),j,i)$, and the rest follows since
$C(S_q)$ has minimum distance $\delta=2$.
\end{proof}

Lemma~\ref{sqnt} is a consequence of Lemma 5.1.1.5 in \cite{neilphd}, but we
give a short proof here for completeness.

\begin{symmetric}\label{sqnt}
$C(S_q)$ is $(\Diag_q(S_q)\rtimes L)$-neighbour transitive.
\end{symmetric}

\begin{proof} By (\ref{act}), $C(S_q)$ is fixed setwise by
$X=\Diag_q(S_q)\rtimes L$. Let $g_1,g_2\in S_q$, and let $y=g_1^{-1}g_2$. Then
$\alpha(g_1)^{x_y}=\alpha(g_1 g_1^{-1} g_2)=\alpha(g_2)$, by (\ref{act}) (i),
and it follows that $X$ is transitive on $C(S_q)$. Now let $i_1\neq j_1$,
$i_2\neq j_2$. Since $S_q$ acts 2-transitively on $Q$, there exists $z\in S_q$
such that $i_1^z=i_2$ and $j_1^z=j_2$. Let $y=g_1^{-1}z g_2$. Then, again using
(\ref{act}), $\nu(\alpha(g_1),i_1,j_1)^{x_y\sigma(z)}=\nu(\alpha(z^{-1}g_1
g_1^{-1}zg_2),i_1^z,j_1^z)= \nu(\alpha(g_2),i_2,j_2)$. Therefore, $X$ acts
transitively on $\Gamma_1(C(S_q))$ and so $C(S_q)$ is $X$-neighbour transitive.
\end{proof}

\begin{notfixed}\label{notfix}
Let $x=x_y \sigma(z)\in\Diag_q(S_q)\rtimes L$. Then $C(A_q)^x=C(z^{-1}yA_q)$. In
particular $C(A_q)^x=C(A_q)$ if and only if $z^{-1}y\in A_q$.
\end{notfixed}

\begin{proof}
By (\ref{act}), $\alpha(g)^x=\alpha(z^{-1}gy)=\alpha(z^{-1}yy^{-1}gy)$ for all $g\in A_q$.  As
$A_q$ is a normal subgroup of $S_q$, the assertion follows.
\end{proof}

\begin{elusive1}\label{elusive1}
$(C(A_q),\Diag_q(S_q)\rtimes L)$ is an elusive pair with parameters $(q,q,3)$.
\end{elusive1}

\begin{proof}
By Lemma \ref{sqnt}, $X=\Diag_q(S_q)\rtimes L$ is transitive on
$\Gamma_1(C(S_q))$ and so, by Lemma~\ref{same}, $X$ is transitive on
$\Gamma_1(C(A_q))$. By Lemma~\ref{notfix}, $C(A_q)$ is not fixed by $X$. Thus
$(C(A_q),X)$ is an elusive pair.
\end{proof} 

For $C(A_q)$, $m=q$, so $m(q-1)$ is even, and also $\delta=3$, as mentioned in
Section~\ref{permcodes}. So $C(A_q)$, indeed satisfies (2). By
Lemma~\ref{notfix}, each element of $\Diag_q(S_q)\rtimes L$ either fixes both
$C(A_q)$ and $C(S_q\setminus A_q)$ setwise, or swaps them.  Note also that
$C(A_q)\cup C(S_q\setminus A_q)=C(S_q)$.

\subsection{Example 2.}\label{ex2}

The \textit{product construction} of a code $C$ in $H(m,q)$, is defined in
\cite[Section 4.7]{neilphd} as follows:
\begin{equation*}
\Prod(C,l)=\{(\alpha_1,\ldots ,\alpha_l)| \alpha_i \in C, \forall i\},
\end{equation*}
which is a code in $H(lm,q)$. We use this construction for the next family of
examples. First we set up the required notation. 

Previously we used a subscript to refer simply to the $k$-th entry of a vertex,
however there is now some ambiguity. We may wish to refer to the $k$-th entry of
a vertex in $H(lm,q)$, or the $k$-th entry, $\alpha_k$, of the $l$-tuple
$(\alpha_1,\ldots,\alpha_l)\in H(lm,q)$, which is itself a vertex in $H(m,q)$.
In this section we always mean the $k$-th entry in $(\alpha_1,\ldots,\alpha_l)$,
so that ``$k$-th entries'' are vertices of $H(m,q)$.

Let $C\subseteq H(m,q)$. By \cite[Lemma 4.7.1]{neilphd}, $C$ and $\Prod(C,l)$
have the same minimum distance, $\delta$ say. If $\delta\geq 2$ then, given
$\boldsymbol{\alpha}=(\alpha_1,\ldots,\alpha_l)\in \Prod(C,l)$, replacing a
single $\alpha_i$ with $\nu\in\Gamma_1(\alpha_i)$ yields a neighbour of
$\boldsymbol\alpha$, which we denote by $\mu(\boldsymbol\alpha,\nu,i)$, where
\begin{equation*}
\mu(\boldsymbol\alpha,\nu,i)|_k=
\left\{
\begin{array}{ll}
\alpha_k & \text{if } k\neq i\\
\nu& \text{if } k=i.
\end{array}\right. 
\end{equation*}
 There are $m(q-1)$ choices for $\nu$, and $l$ choices for $i$, and so all the
$lm(q-1)$ neighbours of $\boldsymbol\alpha$ have this form. Given an action of
$X$ on $\Gamma=H(m,q)$, we can define an action of $X\wr S_l$ on the cartesian
product of $l$ copies of $\Gamma$. Let
$\boldsymbol\alpha=(\alpha_1,\ldots,\alpha_l)$, with each $\alpha_i\in\Gamma$,
$(x_1,\ldots,x_l)\in X^l$, $\sigma\in S_l$. Then
\begin{equation}\label{productaction}
\boldsymbol\alpha^{(x_1,\ldots,x_l)}=(\alpha_1^{x_1},\ldots,\alpha_l^{x_l}),
\quad\text{and}\quad
\boldsymbol\alpha^\sigma=(\alpha_{1^{\sigma^{-1}}},\ldots,\alpha_{l^{\sigma^{-1}
}}),
\end{equation} 
and these elements act on the neighbours of $\Prod(C,l)$ as follows:
\begin{align*}
\mu(\boldsymbol\alpha,\nu,i)^{(x_1,\ldots,x_l)}|_k
&=\left\{
\begin{array}{ll}
\alpha_k^{x_k} & \text{if } k\neq i\\
\nu^{x_i}& \text{if } k=i
\end{array}\right. \\
&=\mu(\boldsymbol\alpha^{(x_1,\ldots,x_l)},\nu^{x_i},i)|_k\text{ }.
\end{align*}
Suppose $k^\sigma=n$. Then,
\begin{align*}
\mu(\boldsymbol\alpha,\nu,i)^\sigma|_n=\mu(\boldsymbol\alpha,\nu,i)|_k
&=\left\{
\begin{array}{ll}
\alpha_k & \text{if } k\neq i\\
\nu\phantom{_k}& \text{if } k=i
\end{array}\right. \\
&=\left\{
\begin{array}{ll}
\alpha_{n^{\sigma^{-1}}} & \text{if } n^{\sigma^{-1}}\neq i\\
\nu\phantom{_{l^{\sigma^{-1}}}}& \text{if } n^{\sigma^{-1}}=i
\end{array}\right. \\
&=\left\{
\begin{array}{ll}
\alpha_{n^{\sigma^{-1}}} & \text{if } n\neq i^\sigma\\
\nu\phantom{_{l^{\sigma^{-1}}}}& \text{if } n=i^\sigma
\end{array}\right. \\
&=\mu(\boldsymbol\alpha^\sigma,\nu,i^\sigma)|_n\text{ }.
\end{align*}
Which gives 
\begin{equation*}\label{base}
\mu(\boldsymbol\alpha,\nu,i)^{(x_1,\ldots,x_l)}=\mu(\boldsymbol\alpha^{(x_1,
\ldots,x_l)},\nu^{x_i},i),\quad\text{and}\quad \mu(\boldsymbol\alpha,\nu,i)^\sigma=\mu(\boldsymbol\alpha^\sigma,\nu,i^\sigma).
\end{equation*}

If $C$ has minimum distance $\delta\geq 3$ in $H(m,q)$, then each neighbour of $\Prod(C,\ell)$
has a unique representation of the form $\mu(\boldsymbol\alpha,\nu,i)$. This, however, is not
the case when $\delta\leq 2$.  Let $C$ be a code with $\delta=2$ and $\alpha,\beta\in C$ such that
$d(\alpha,\beta)=2$, and consider $\nu\in\Gamma_1(\alpha)\cap\Gamma_1(\beta)$.  
Then, for $\boldsymbol\alpha=(\alpha,\ldots,\alpha)$ and
$\boldsymbol\beta=(\beta,\alpha,\ldots,\alpha)$ in $\Prod(C,l)$, it follows that 
$\mu(\boldsymbol\alpha,\nu,1)=\mu(\boldsymbol\beta,\nu,1)$.  Gillespie \cite[Lemma 4.7.3]{neilphd} 
proved the next result for codes with $\delta\geq 3$, however it is in 
fact true for arbitrary minimum distance.  

\begin{ntproduct}\label{ntprod} Let $C$ be an $X$-neighbour transitive code
in $H(m,q)$.  Then $\Prod(C,l)$ is $X\wr S_l$-neighbour transitive in $H(lm,q)$.
\end{ntproduct}

\begin{proof}
It follows from (\ref{productaction}) that $X^l$ is transitive on $\Prod(C,l)$
since $X$ is transitive on $C$. To map the neighbour
$\mu(\boldsymbol\alpha,\nu,i)$ to the neighbour
$\mu(\boldsymbol\beta,\nu',j)$, we first apply $\sigma=(ij)\in S_l$, so
$\mu(\boldsymbol\alpha,\nu,i)^\sigma=\mu(\boldsymbol\alpha^\sigma,\nu,j)$.  As $C$ and $\Gamma_1(C)$
are both $X$-orbits in $H(m,q)$, there exists $x_k\in X$ such that $\alpha_k^{x_k}=\beta_k$
for $k\neq i,j$, there exists $x_i\in X$ such that $\alpha_j^{x_i}=\beta_i$, and there exists $x_j\in X$
such that $\nu^{x_j}=\nu'$.  By letting $x=(x_1,\ldots,x_l)\in X^l$, it follows that 
$\mu(\boldsymbol\alpha,\nu,i)^{\sigma x}=\mu(\boldsymbol\beta,\nu',j)$.
\end{proof}

The next result follows directly from Lemmas~\ref{sqnt} and
\ref{ntprod}.

\begin{prodsq}\label{psq}
$\Prod(C(S_q),l)$ is $(\Diag_q(S_q)\rtimes L)\wr S_l$-neighbour transitive.
\end{prodsq}

\begin{cql}\label{cql}
Let $C(q,l)$ be the subset of $\Prod(C(S_q),l)$ where
$(\alpha(g_1),\ldots,\alpha(g_l))\in C(q,l)$ if and only if $|\{i\mid g_i\in
A_q\}|$ is even. 
\end{cql}

In the remainder of the section we show that $(C(q,l),(\Diag_q(S_q)\rtimes L)\wr
S_l)$ is an elusive pair. 

\begin{sqneighbours}\label{sqn}
$\Gamma_1(\Prod(C(S_q),l))=\Gamma_1(C(q,l))$, with $C(q,l)$ as in
Definition~\ref{cql}.
\end{sqneighbours}

\begin{proof}
Set $\mathcal{P}=\Prod(C(S_q),l)$ and $\mathcal{C}=C(q,l)$. Let
$\boldsymbol\alpha=(\alpha(g_1),\ldots,\alpha(g_l))\in\mathcal{P}$,
$\nu=\nu(\alpha(g_k),i,j)$ for some $i\neq j\leq q$, and
$\boldsymbol\mu=\mu(\boldsymbol\alpha,\nu,k)\in\Gamma_1(\mathcal{P})$.  
Suppose $\boldsymbol\alpha\notin\mathcal{C}$, and let $g_n'=g_n$ for $n\neq k$ and
$g_k'=(ij)g_k$. Since $g_k$ and $g_k'$ have different parities, it follows that
$\boldsymbol\alpha'=(\alpha(g_1'),\ldots,\alpha(g_l'))\in\mathcal{C}$. Consider
$\nu'=\nu(\alpha(g_k'),j,i)$ and
$\boldsymbol\mu'=\mu(\boldsymbol\alpha',\nu',k)$. By Lemma~\ref{same},
$\nu(\alpha(g),i,j)=\nu(\alpha(g'),j,i) \in H(q,q)$. Thus, 
$\boldsymbol\mu=\boldsymbol\mu'\in\Gamma_1(\mathcal{C})$, so
$\Gamma_1(\mathcal{P})\subseteq \Gamma_1(\mathcal{C})$. The fact that
$\Gamma_1(\mathcal{C})\subseteq\Gamma_1(\mathcal{P})$ holds because
$\mathcal{C}\subseteq \mathcal{P}$ and $\mathcal{P}$ has minimum distance $2$.
\end{proof}

\begin{pair2}\label{pair2}
$(C(q,l),(\Diag_q(S_q)\rtimes L)\wr S_l)$ is an elusive pair, with parameters
$(lq,q,3)$.
\end{pair2}

\begin{proof}
By Corollary~\ref{psq}, $X=(\Diag_q(S_q)\rtimes L)\wr S_l$ is transitive on
$\Gamma_1(\Prod(C(S_q),l))$, and this set is equal to $\Gamma_1(C(q,l))$, by
Lemma~\ref{sqn}. We now show that $C(q,l)$ is not fixed by $X$. Consider
$\boldsymbol\alpha=(\alpha(1),\ldots,\alpha(1))\in C(q,l)$ and the element
$x=(x_y,1,\ldots,1) $ in the base group of $(\Diag_q(S_q)\rtimes L)\wr S_l$,
where $y=(1 2)\in S_q$. Then $\boldsymbol\alpha^x=(\alpha( (1
2)),\alpha(1),\ldots,\alpha(1))\notin C(q,l)$. Thus $(C(q,l),X)$ is an elusive
pair. 

It remains to show that $C(q,l)$ has minimum distance $3$. Let
$\boldsymbol\alpha=(\alpha(g_1),\ldots,\alpha(g_l))$ and
$\boldsymbol\beta=(\alpha(g'_1),\ldots,\alpha(g'_l)) \in C(q,l)$, with
$\boldsymbol\alpha\neq \boldsymbol\beta$. If there exists $i\neq j$ such that
$g_i\neq g_i'$ and $g_j\neq g_j'$, then
$d(\boldsymbol\alpha,\boldsymbol\beta)\geq 2 \delta_{C(S_q)}=4$. So suppose
there exists $i$ such that $g_k=g_k'$ for all $k\neq i$. Then $g_i$ and $g_i'$
have the same parity, and so are either both from $C(A_q)$ or both from
$C(S_q\setminus A_q)$. Hence $d(\boldsymbol\alpha,\boldsymbol\beta)\geq
\delta_{C(A_q)}=3$. For equality set $g_i=1$ and $g_i'=(123)$.
\end{proof}

This completes the proof of Theorem~\ref{alternating}, as
$(C(q,l),(\Diag_q(S_q)\rtimes L)\wr S_l)$ has parameters $(lq,q,3)$. Note that
each element of $(\Diag_q(S_q)\rtimes L)\wr S_l$ either fixes $C(q,l)$ setwise,
or sends it to the code $C'(q,l)\subseteq \Prod(C(S_q),l)$, where
$\boldsymbol\alpha=(\alpha(g_1),\ldots,\alpha(g_l))\in C'(q,l)$ if and only if
$|\{i\mid g_i\in A_q\}|$ is odd. Observe also that $C(q,l)\cup C'(q,l)=\Prod(C(S_q),l)$.

\subsection{Example 3.}\label{except}

For $a\in Q$ we define $\beta(a)=(a,\ldots,a)\in H(m,q)$. We define the
\emph{repetition code} in $H(m,q)$ as 
\begin{equation*}
\Rep(m,q)=\{\beta(a)\mid a\in Q\}.
\end{equation*}
By \cite[Theorem 3.2]{fpas}, $\Rep(m,q)$ is
$\Diag_m(S_q)\rtimes L$-neighbour transitive with minimum distance $m$. We now construct our final
example, which does not share some of the properties of previous examples.

\begin{altrep}\label{altrep}
Let $C=C(A_q)\cup \Rep(q,q)$, $X=\Diag_q(S_q)\rtimes L$ and $q\geq 4$. Then
$(C,X)$ is an elusive pair and $X$ is not transitive on $\Gamma_1(C)$. Moreover,
for $q\geq 5$, $X$ is the setwise stabiliser in $\Aut(\Gamma)$ of $\Gamma_1(C)$.
\end{altrep}

\begin{proof}
Let $R=\Rep(q,q)$. Note that $\delta_{C(A_q)}=3$ and $\delta_{R}=q$. If
$\alpha\in C(A_q)$ and $\beta\in R$ then $d(\alpha,\beta)=q-1\geq 3$, so
$\Gamma_1(C)=\Gamma_1(C(A_q))\cup\Gamma_1(R)$. By Lemma~\ref{same} and Corollary~\ref{elusive1}, $X$
is transitive on $\Gamma_1(C(A_q))$ and, as mentioned above, $X$ is also
transitive on $\Gamma_1(R)$. In particular, $X$ fixes $\Gamma_1(C)$ setwise. It
follows from Lemma~\ref{notfix} that any element of $X$ either
fixes $C$ setwise, or sends it to $C'=C(S_q\setminus A_q)\cup R$. Thus $(C,X)$
is an elusive pair with parameters $(q,q,3)$. There are, however, two $X$-orbits
in $\Gamma_1(C)$, so $X$ is not transitive on $\Gamma_1(C)$.

Let $G=\Aut(\Gamma)_{\Gamma_1(C)}$, the setwise stabiliser in $\Aut(\Gamma)$ of
$\Gamma_1(C)$. We now show that $X=G$ when $q\geq 5$. By \cite{ntrcodes}, since
$\delta_{R}\geq5$, $\Aut(\Gamma)_{\Gamma_1(R)}=\Aut(R)$ and, by \cite[Theorem
3.2]{fpas}, $\Aut(R)=X$. Suppose there exists $x\in G\setminus X$. Then because
$\Gamma_1(C(A_q))$ and $\Gamma_1(R)$ are both $X$-orbits, it follows that $G$
acts transitively on $\Gamma_1(C)$.  Therefore, the number, $|\Gamma_1(\mu)\cap \Gamma_1(C)|$, of neighbours of the code adjacent to
$\mu\in\Gamma_1(C)$ is independent of the choice of $\mu$.  

Now we inspect the neighbours of $\mu=(1,1,3,4,\ldots)\in\Gamma_1(C)$. Changing
the first entry to 2 gives us a vertex in $\Gamma_2(C)$, but the other $q-2$ choices give
us a vertex in $\Gamma_1(C)$. Changing the second entry to 2 gives us the codeword $\alpha(1)$, 
however the other $q-2$ choices give vertices in $\Gamma_1(C)$. For $3\leq
i\leq q$, replacing the $i$-th entry with 2 gives us a vertex in $\Gamma_1(C)$,
while the other $q-2$ choices give vertices in $\Gamma_2(C)$. Thus
$|\Gamma_1(\mu)\cap \Gamma_1(C)|=3(q-2)$. Now let
$\nu=(2,1,1,1,\ldots)\in\Gamma_1(C)$. The adjacent vertex with 1 in the first
entry is in $C$, but the $q-2$ other vertices that differ in the first entry are
in $\Gamma_1(C)$. Changing any other entry gives a vertex which is always in
$\Gamma_2(C)$, since $q\geq 5$.  Thus $|\Gamma_1(\nu)\cap \Gamma_1(C)|=q-2\neq
|\Gamma_1(\mu)\cap \Gamma_1(C)|$, which is a contradiction.
\end{proof}

In the case $q=4$, let $h=(13)(24)$ and $x=(1,1,h,h)\in \Aut(\Gamma)$. A
straightforward, but somewhat lengthy, calculation shows that $x$ fixes
$\Gamma_1(C)$ and maps the vertex $(1,1,3,4)\in\Gamma_1(C(A_q))$ to
$(1,1,1,2)\in\Gamma_1(\Rep(q,q))$.  In this case, $(C,X)$ is an elusive pair 
for the group $X=\langle \Diag_4(S_4)\rtimes L,x\rangle$, but $X$ acts transitively on $\Gamma_1(C)$.  

The first section of the proof of Proposition~\ref{altrep} also shows that the
image of $C$ under any $x\in X$ that does not fix $C$ is $C'=C(S_q\setminus
A_q)\cup \Rep(q,q)$, and we note that $C\cap  C'=\Rep(q,q)\neq \phi$.

\subsection{Non-Existence of Elusive Codes with Parameters
$(4,3,3)$.}\label{case}

Now we proceed to show that it is not possible to have an elusive code of length
four, with minimum distance three and an alphabet size three. First we introduce
some results and notation from \cite{ntrcodes}. 

We say that two codes, $C$ and $C'$, in $H(m,q)$, are \textit{equivalent} if
there exists $y\in \Aut(\Gamma)$ such that $C^y=C'$. Equivalence preserves
minimum distance. (See \cite[Lemma 4]{ntrcodes}). The next lemma is an extension
of \cite[Lemma 1]{ntrcodes}.

\begin{neigh}\label{neigh}
If $\alpha$ and $\beta$ are in $H(m,q)$ with $d(\alpha,\beta)=2$, then there are
precisely two distinct vertices, $\mu$ and $\nu$, in
$\Gamma_1(\alpha)\cap\Gamma_1(\beta)$, and $d(\mu,\nu)=2$. Moreover, given
$\alpha,\mu$ and $\nu$ there is only one possible choice for $\beta$.
\end{neigh}

\begin{proof}
By \cite[Lemma 1]{ntrcodes}, $|\Gamma_1(\alpha)\cap\Gamma_1(\beta)|=2$. We know
$\mu,\nu\in\Gamma_1(\alpha)\cap\Gamma_1(\beta)$ each differ from $\alpha$ in one
entry, say $\mu_i\neq\alpha_i$ and $\nu_j\neq\alpha_j$. If $i\neq j$, then
$d(\mu,\nu)=2$. Suppose $i=j$. Then $\beta_i$ is not equal to at least one of
$\mu_i$ or $\nu_i$, since $\mu_i\neq\nu_i$. We know $d(\beta,\mu)=1$, and so if
$\beta_i\neq\mu_i$ then $\beta_l=\mu_l=\alpha_l$ for $l\neq i$, a contradiction
since $d(\alpha,\beta)=2$. A similar argument rules out the case $\beta_i\neq
\nu_i$, and we are left with $d(\mu,\nu)=2$ and $\beta_i=\mu_i, \beta_j=\nu_j$
and $\beta_l=\alpha_l$ for $l\neq i,j$.
\end{proof}

Let $(C,X)$ be an elusive pair in $H(m,q)$ with $\delta\geq 3$. Suppose
$\alpha\in C$ and $x\in X$ such that $\alpha^x\notin C$. A \textit{pre-codeword}
of $\alpha$ with respect to $x$ is a vertex $\pi$ such that $d(\alpha,\pi)=2$
and $\pi^x\in C$ \cite[Definition 3]{ntrcodes}. We denote the set of all
pre-codewords of $\alpha$ with respect to $x$ by $\Pre(\alpha,x)$. 

\begin{partition}\label{partition}
Let $(C,X)$ be an elusive pair in $H(m,q)$ with $\delta\geq 3$, ${\alpha}\in C$,
$x\in X$ such that $\alpha^x\notin C$, and $\pi\in\Pre(\alpha,x)$. Then 
\begin{itemize}
\item[(i)]{$\{\Gamma_1(\alpha)\cap\Gamma_1(\pi')\mid \pi'\in\Pre(\alpha,x)\}$
forms a partition of $\Gamma_1(\alpha)$.}
\item[(ii)]{$\{\Gamma_1(\pi)\cap\Gamma_1(\beta)\mid \beta\in\Gamma_2(\pi)\cap
C\}$ forms a partition of $\Gamma_1(\pi)$.}
\item[(iii)]{Let $\mu,\nu\in \Gamma_1(\alpha)\cap\Gamma_1(\pi)$ differ from
$\alpha$ in entries $i,j$ respectively, and $\pi' \in \Pre(\alpha,x)\setminus
\pi$. Then there exists  $\nu'\in \Gamma_1(\alpha)\cap\Gamma_1(\pi')$ that
differs from $\alpha$ in some entry $k\neq i,j$. }
\end{itemize}
\end{partition}

\begin{proof}
For a proof of (i) see \cite[Lemma 6 (i)]{ntrcodes}, and of (ii) see \cite[Lemma
7 (ii)]{ntrcodes}. 

For part (iii), $d(\alpha,\pi')=2$, so $\alpha$ and $\pi'$ differ in exactly two
entries $i'$ and $j'$. If $\{i,j\}=\{i',j'\}$, then $d(\pi,\pi')=1$ or $2$,
since $\pi\neq\pi'$. However $d(\pi,\pi')=d(\pi^x,{\pi'}^x) \geq 3$, since
$\pi^x,{\pi'}^x\in C$ and $\delta\geq 3$. So there must be some value
$k\notin\{i,j\}$ and then, by Lemma~\ref{neigh}, vertex $\nu'\in
\Gamma_1(\alpha)\cap\Gamma_1(\pi')$ with $\nu_l'=\alpha_l$, for $l\neq k$, and
$\nu_k'=\pi_k'$.
\end{proof}

Note that, by \cite[Lemma 1]{ntrcodes}, each part in the above partitions has
size $2$, since the distance between a codeword and pre-codeword is $2$. The
partitions themselves have size $m(q-1)/2$, by \cite[Lemma 6 (ii)]{ntrcodes} and
\cite[Lemma 7 (iii)]{ntrcodes}.

\begin{q3m4}\label{q3m4}
There is no elusive pair with parameters $(4,3,3)$. 
\end{q3m4}

\begin{proof}
Let $(C,X)$ be an elusive pair with parameters $(4,3,3)$. By replacing $C$ with
an equivalent code if necessary, we can assume that ${\bf{0}}=0000\in C$ and
that there exists $x\in X$ such that ${\bf{0}}^x\notin C$. First we determine, up to
equivalence, four members of $\Pre(\boldsymbol{0},x)$. By \cite{ntrcodes} it
follows that $|\Pre({\bf{0}},x)|=4$ and that $\Gamma_1(\pi)\subseteq\Gamma_1(C)$ for each $\pi\in\Pre({\bf{0}},x)$.  By
Lemma \ref{partition} (i), $\mathcal{P}_0=\{\Gamma_1({\bf{0}})\cap\Gamma_1(\pi)\mid\pi\in\Pre({\bf{0}},x)\}$ 
forms a partition of $\Gamma_1({\bf{0}})$, and by Lemma~\ref{neigh}, each part of this partition
consists of two vertices.

Consider $\{1000,\nu_1\}\in\mathcal{P}_0$.  By Lemma \ref{neigh}, $\nu_1\neq
2000$.  Thus, by replacing
$C$ with an equivalent code if necessary, we can assume that $\nu_1=0100$ and,
again by Lemma~\ref{neigh}, $\pi_1=1100\in\Pre({\bf{0}},x)$.
Next, consider $\{2000,\nu_2\}\in\mathcal{P}_0$.  By Lemma \ref{partition}
(iii), $\nu_2\neq 0200$, and so again,
using the symmetries of the Hamming graph, we can assume that $\nu_2=0020$. 
Therefore $\pi_2=2020\in\Pre({\bf{0}},x)$.
Next, consider $\{0200,\nu_3\}\in\mathcal{P}_0$.  If $\nu_3=0010$, this implies
that $\{0001,0002\}\in\mathcal{P}_0$,
contradicting Lemma \ref{neigh}.  Thus, as before, we can assume that
$\nu_3=0002$ and $\pi_3=0202\in\Pre({\bf{0}},x)$.
Consequently we deduce that $\pi_4=0011\in\Pre({\bf{0}},x)$.

Next we determine three additional elements of $C$. By Lemma \ref{partition}
(ii), $\mathcal{P}_1=\{\Gamma_1(\pi_1)\cap\Gamma_1(\alpha)\mid\alpha\in\Gamma_2(\pi_1)\cap C\}$
forms a partition of $$\Gamma_1(\pi_1)=\left\{\begin{array}{llll}
0100 & 1000 & 1110 & 1101 \\
2100 & 1200 & 1120 & 1102
\end{array}\right\},$$ and by Lemma~\ref{neigh},
each part has size $2$.  We know that $\{1000,0100\}\in\mathcal{P}_1$ as
${\bf{0}}\in\Gamma_2(\pi_1)\cap C$.  Furthermore, by Lemma \ref{neigh},
$\{1110,1120\}$ and $\{1101,1102\}$ are not elements of $\mathcal{P}_1$.  This
implies that at least one of $1110$
or $1120$ forms an element of $\mathcal{P}_1$ with either $1101$ or $1102$. 
Thus
at least one of $1111$, $1121$, $1122$ or $1112$ is a codeword
in $\Gamma_1(\pi_1)\cap C$.  Consider $1020\in\Gamma_1(\pi_2)$, which must
be adjacent to a codeword with three non-zero entries, as $\delta=3$ and
${\bf{0}}\in C$. Such
a codeword has the form $1a2b$, and is at distance at least $2$ from $1121$ and $1122$,
so these
vertices are not codewords. By considering $0102\in\Gamma_1(\pi_3)$,
a similar argument shows that $1112$ is not a codeword either.  Thus $1111\in C$
and $\{1110,1101\}\in\mathcal{P}_1$.
Consider the part $\{2100,\nu\}\in\mathcal{P}_1$. By Lemma \ref{partition}
(iii), $\nu\neq 1200$, and if
$\nu=1120$ then $2120\in C$, contradicting the fact that
$2120\in\Gamma_1(\pi_2)$.  Thus $\nu=1102$, which
leaves $\{1200,1120\}\in\mathcal{P}_1$.  Hence $\Gamma_2(\pi_1)\cap
C=\{0000,1111,1220,2101\}$.
Finally, consider the partition
$\mathcal{P}_2=\{\Gamma_1(\pi_2)\cap\Gamma_1(\alpha)\mid\alpha\in\Gamma_2(\pi_2)\cap C\}$
of $$\Gamma_1(\pi_2)=\left\{
\begin{array}{llll}
0020 & 2120 & 2000 & 2021 \\
1020 & 2220 & 2010 & 2022
\end{array}\right\}.$$
As ${\bf{0}}$, $1220\in\Gamma_2(\pi_2)\cap C$ it follows that $\{0020,2000\}$,
$\{1020,2220\}\in\mathcal{P}_2$.  Consider
the part $\{2021,\mu\}\in\mathcal{P}_2$.  By Lemma \ref{neigh}, $\mu\neq 2022$. 
Thus $\mu=2120$ or $2010$ and so $\alpha=2121$ or $2011\in C$
respectively.  However, in both cases $d(\alpha,1111)=2$, which is a
contradiction.  Thus no such elusive pair exists.
\end{proof}

\section[]{acknowledgements} The second author is supported by the Australian Research Council Federation Fellowship FF0776186 of the third author.

\end{document}